\documentclass{article}
\usepackage{graphicx} 

\usepackage{amsmath,amssymb,amsthm,xcolor,url,mathtools}
\usepackage{cancel}
\usepackage{hyperref,comment}

\usepackage{mathabx}

\usepackage{enumerate}

\newtheorem{theorem}{Theorem}

\newtheorem{proposition}[theorem]{Proposition}

\theoremstyle{definition}
\newtheorem{definition}[theorem]{Definition}
\theoremstyle{remark}
\newtheorem{remark}{Remark}

\DeclareMathOperator*{\argmin}{arg\,min} 
\DeclareMathOperator*{\argmax}{arg\,max} 
\DeclareMathOperator*{\argminimax}{arg\,minimax} 

\usepackage{natbib}
\title{A Bregman firmly nonexpansive proximal operator for baryconvex optimization}
\author{Mastane Achab\thanks{\texttt{mastane@deepgambit.com}}\\
Deep Gambit Limited, Masdar City, Abu Dhabi, UAE\\
\href{https://www.deepgambit.com}{www.deepgambit.com}}
\date{}

\begin{document}

\maketitle

\begin{abstract}
We present a generalization of the proximal operator defined through a convex combination of convex objectives, where the coefficients are updated in a minimax fashion. We prove that this new operator is Bregman firmly nonexpansive with respect to a Bregman divergence that combines Euclidean and information geometries; and that its fixed points are given by the critical points of a certain nonconvex function. Finally, we derive the associated continuous flows.
\end{abstract}

\paragraph{Notations}
The Euclidean norm of any vector $x \in \mathbb{R}^m$ ($m\ge 1$) is denoted $\| x \|$.
For any integer $S\ge 2$, we denote by $\mathbf{1}_S$ the all-ones vector of size $S$ and by $\Delta_S$ the probability simplex:
\begin{equation*}
    \Delta_S = \left\{ q=(q_1,\dots,q_S) \in [0,1]^S : q_1 +\dots +q_S = 1 \right\} \ .
\end{equation*}
The Kullback-Leibler divergence \cite{kullback1951information} will be denoted by ``$D_{\text{KL}}$'' throughout the paper: for any $q, r \in \ring \Delta_S$,
$ D_{\text{KL}}( r \| q )
= \sum_{s=1}^S r_s \log\left( \frac{r_s}{q_s} \right) $.
Let $h(q)=\sum_{s=1}^S q_s \log(q_s)$ be the negative entropy function defined over $\ring \Delta_S$; its gradient $\nabla h(q)=(1+\log(q_s))_s$,
and the softargmax function
\begin{equation*}
\sigma : \xi=(\xi_1,\dots,\xi_S) \in \mathbb{R}^S \mapsto (\nabla h)^{-1}(\xi-\log(\sum_{s=1}^S e^{\xi_s-1})\mathbf{1}_S) .
\end{equation*}
Given a differentiable function $\ell=(\ell_1,\dots,\ell_S) : \mathbb{R}^m \rightarrow \mathbb{R}^S$, we denote by $J_\ell$ its Jacobian matrix.
Finally, given $(x,q) \in \mathbb{R}^m \times \Delta_S$, we refer to the vector $J_\ell(x)^\intercal q = \sum_{s=1}^S q_s \nabla \ell_s(x)$ as the ``$q$-barygradient of $\ell$ at $x$''.

\section{Proximal operator}
\label{sec:intro}

In this article we present a generalization of the convex optimization formalism (\cite{boyd2004convex}) that we call \emph{baryconvex optimization} since it involves weighted convex objectives where the weights are learned in a minimax fashion.

\begin{definition}[Generalized proximal operator]
\label{def:prox}
    Let $\ell = (\ell_1,\dots,\ell_S) : \mathbb{R}^m \rightarrow \mathbb{R}^S$ for $m \ge 1, S \ge 2$ and where $\ell_s$ is a differentiable convex function for each $s \in \{1,\dots,S\}$.
    Given $\lambda >0$, we define our generalized proximal operator $\text{prox}_{\lambda \ell} $ as follows:
    for all $(x,q) \in \mathbb{R}^m \times \ring \Delta_S$
    \begin{equation*}
        \text{prox}_{\lambda \ell}(x,q)
        = \argminimax_{(z,r) \in \mathbb{R}^m \times \ring \Delta_S} H_{x,q}(z,r) :=
    r^\intercal \ell(z) 
    + \frac{1}{2\lambda} \| z-x \|^2 - \frac{1}{\lambda} D_{\text{KL}}( r \| q) \ .
    \end{equation*}
\end{definition}

First notice that in the degenerate case $S=1$, the probability simplex is reduced to the singleton $\Delta_1=\{1\}$, and we recover the standard proximal operator with a single convex loss function whose minimizers are exactly the fixed points of the prox.
This paper proposes to extend well-known convex optimization methods such as the proximal point algorithm (PPA, see \cite{rockafellar1976monotone}) and gradient descent (GD, see \cite{boyd2004convex}) to our general setting with $S\ge 2$.
\begin{quote}
\textit{Question: Can we compute a fixed point (if it exists) of the generalized prox in Definition \ref{def:prox}?}
\end{quote}
As will be shown, the answer provided by this paper is positive.
\begin{quote}
\textit{Answer: Yes, by leveraging a Bregman geometry that combines Euclidean and simplex structures.}
\end{quote}

\paragraph{Saddle point}
We point out that the function $(z,r)\mapsto H_{x,q}(z, r)$ is strongly convex-concave (i.e. strongly convex in $z$ and strongly concave in $r$, see e.g. \cite{boyd2004convex}) and admits a unique saddle point $(x',q')=\text{prox}_{\lambda \ell}(x,q)$ characterized by the stationarity condition 
\begin{equation}
\label{eq:statioCond}
\nabla H_{x,q}(x',q')= \begin{pmatrix} \mathbf{0}_m \\
-\frac{c}{\lambda} \mathbf{1}_S \end{pmatrix} \quad \text{with } c = \log(\sum_s e^{\log(q'_s)-\lambda \ell_s(x')}) -\log(\sum_s e^{\log(q_s)}) .
\end{equation}
%
Further, by the \emph{minimax theorem}\footnote{see e.g. \href{https://en.wikipedia.org/wiki/Minimax_theorem}{wikipedia.org/Minimax\_theorem} or Theorem 7.1 in \cite{cesa2006prediction} }, we have:
\begin{equation}
\label{eq:minimax}
    \min_z \max_{r} H_{x,q}(z,r) = H_{x,q}(x', q' )
    =  \max_{r} \min_z H_{x,q}(z,r) .
\end{equation}

\paragraph{Tensorization}
Consider two generalized proximal operators (with same step-size $\lambda>0$) characterized for $i\in \{1,2\}$ by $\ell^{(i)}:\mathbb{R}^m \rightarrow \mathbb{R}^{S_i}$ with $S_i \ge 2$. We can define the outer sum $\ell^{(1)}\oplus \ell^{(2)} : \mathbb{R}^m \rightarrow \mathbb{R}^{S_1\times S_2}$ as $[\ell^{(1)}\oplus \ell^{(2)}]_{j,k}=\ell^{(1)}_j+\ell^{(2)}_k$, and for probabilities $q^{(i)} \in \Delta_{S_i}$ the outer product $q^{(1)} \otimes q^{(2)} \in \Delta_{S_1\times S_2}$ as $[q^{(1)} \otimes q^{(2)}]_{j,k} = q^{(1)}_j q^{(2)}_k$, for all $(j,k) \in \{1,\dots,S_1\} \times \{1,\dots,S_2\}$.
We have the following property concerning the proximity operator of $\ell^{(1)}\oplus\ell^{(2)}$.



\begin{proposition}[Closure property]
\label{prop:tensorization3}
If
\begin{equation*}
(x',q') = \text{prox}_{\lambda \ell^{(1)}\oplus \ell^{(2)}}(x,q^{(1)}\otimes q^{(2)}) \, ,
\end{equation*}
then there exist $\tilde q^{(i)} \in \Delta_{S_i}$ such that $q' = \tilde q^{(1)} \otimes \tilde q^{(2)}$.
\end{proposition}

The proof of Proposition \ref{prop:tensorization3} is straightforward using Eq. (\ref{eq:statioCond}).
 In other words, the next distribution $q'$ remains an outer product and we deduce that the subset $\mathbb{R}^m \times (\Delta_{S_1}\otimes \Delta_{S_2}) \subseteq \mathbb{R}^m \times \Delta_{S_1\times S_2}$ is closed under $\text{prox}_{\lambda \ell^{(1)}\oplus \ell^{(2)}}$:
 \begin{equation}
     \text{prox}_{\lambda \ell^{(1)}\oplus \ell^{(2)}}(\mathbb{R}^m \times (\Delta_{S_1}\otimes \Delta_{S_2}))
     \subseteq \mathbb{R}^m \times (\Delta_{S_1}\otimes \Delta_{S_2}).
 \end{equation}
 This closure property can even be extended to any multi-index set $\mathcal{T} \subseteq \{1,\dots,S\}^N$ with $[\bigoplus_{k=1}^N \ell^{(k)}]_\tau=\sum_{k=1}^N \ell^{(k)}_{\tau_k}$ and $[\bigotimes_{k=1}^N q^{(k)}]_\tau \propto \prod_{k=1}^N q^{(k)}_{\tau_k}$ for all $\tau=(\tau_1,\dots,\tau_N) \in \mathcal{T}$. In particular, this allows to optimize models that are based on circular convolutions (\cite{achab2022checkered},\cite{achab2023beyond}).

In the next sections, we propose to generalize some key components of the convex analysis toolbox (firm nonexpansion property \cite{bauschkeconvex}, PPA and GD methods) in order to find a fixed point of $\text{prox}_{\lambda \ell}$ in the general case $S\ge 2$.

\section{Bregman firm nonexpansiveness}

We recall from \cite{brohe2000perturbed}-\cite{bauschke2003bregman} that an operator $T$ is Bregman firmly nonexpansive (BFNE) with respect to $f$ if $\langle Tx-Ty, \nabla f(Tx) - \nabla f(Ty) \rangle \le \langle Tx-Ty, \nabla f(x) - \nabla f(y) \rangle$, $\forall x,y$.
Furthermore, if the BFNE operator has a fixed point $x^*=T x^*$, any sequence obtained by recursively applying $T$, namely $x^{k+1}=T x^k$, converges to a fixed point.
Our main result (Theorem \ref{thm:bfne} below) states that our generalized proximal operator introduced in section \ref{sec:intro} is BFNE with respect to a hybrid Bregman divergence mixing the squared Euclidean and the KL divergences.

\begin{definition}[Euclidean+KL Bregman divergence]
\label{def:bregdiv}
    Let the function $f$ be defined for all $(x,q) \in \mathbb{R}^m\times \ring \Delta_S $ as follows:
\begin{equation*}
    f(x, q) = \frac12 \|x\|^2 + h(q)
\end{equation*}
and its corresponding Bregman divergence:
\begin{equation*}
    D_f\left( \begin{pmatrix}
        x\\ q
    \end{pmatrix} , \begin{pmatrix}
        x'\\ q'
    \end{pmatrix} \right)
    = \frac12 \|x-x'\|^2 + D_{\text{KL}}(q\|q').
\end{equation*}
\end{definition}

\begin{theorem}[BFNE]
\label{thm:bfne}
    Let $\text{prox}_{\lambda \ell}$ and $f$ be as defined in Definitions \ref{def:prox} and \ref{def:bregdiv} respectively.
    Then, $\text{prox}_{\lambda \ell}$ is Bregman firmly nonexpansive with respect to $f$.
\end{theorem}

\begin{proof}
For $q\in \ring \Delta_S$, we have by the convexity of $x\mapsto q^\intercal \ell(x)$:
\begin{equation}
\label{eq:cvx_q}
    q^\intercal \ell(z)-q^\intercal \ell(x) \ge q^\intercal J_\ell(x) (z-x)
\end{equation}
and, similarly, for any other $r\in \ring \Delta_S$:
\begin{equation}
\label{eq:cvx_qp}
    r^\intercal \ell(x)-r^\intercal \ell(z) \ge r^\intercal J_\ell(z) (x-z) .
\end{equation}
Then, by summing Eqs. \ref{eq:cvx_q} and \ref{eq:cvx_qp} it holds:
\begin{multline}
\label{eq:Amonotone}
    (J_\ell(z)^\intercal r - J_\ell(x)^\intercal q )^\intercal (z-x)
    \ge q^\intercal \ell(x)-q^\intercal \ell(z)
    + r^\intercal \ell(z) - r^\intercal \ell(x)  \\
    \Longleftrightarrow
    \left\langle \begin{pmatrix}
        z \\ r
    \end{pmatrix} - \begin{pmatrix}
        x \\ q
    \end{pmatrix}  , \begin{pmatrix}
        J_\ell(z)^\intercal r \\ -\ell(z)
    \end{pmatrix} - \begin{pmatrix}
        J_\ell(x)^\intercal q \\ -\ell(x)
    \end{pmatrix}  \right\rangle \ge 0 .
\end{multline}

From the stationarity condition (\ref{eq:statioCond}) satisfied by the saddle point $(x',q')=\text{prox}_{\lambda}(x,q)$ of the function $H_{x,q}$:
\begin{multline}
\label{eq:Hstatio}
    \nabla H_{x,q}(x',q') = \begin{pmatrix}
        \mathbf{0}_m \\
        -\frac{c}{\lambda} \mathbf{1}_S
    \end{pmatrix}
    \Leftrightarrow 
    \begin{cases}
        J_\ell(x')^\intercal q' + \frac{1}{\lambda}(x'-x) = \mathbf{0}_m \\
        \ell(x') - \frac{1}{\lambda}( \nabla h(q') - \nabla h(q) ) = -\frac{c}{\lambda} \mathbf{1}_S
    \end{cases} \\
    \Leftrightarrow
    \begin{cases}
        x = x' + \lambda J_\ell(x')^\intercal q' \\
        \nabla h(q) + c \mathbf{1}_S = \nabla h(q') - \lambda \ell(x') .
    \end{cases}
\end{multline}

We are now ready to prove that $\text{prox}_{\lambda \ell}$ is BFNE w.r.t. $f$. For $x,z \in \mathbb{R}^m$, $q,r\in \ring \Delta_S$ and $(x',q')=\text{prox}_{\lambda \ell}(x,q)$, $(z',r')=\text{prox}_{\lambda \ell}(z,r)$

\begin{multline}
    \langle \text{prox}_{\lambda \ell}(x,q) - \text{prox}_{\lambda \ell}(z,r) , \nabla f(x,q) - \nabla f(z,r) \rangle = \\
    \left\langle
    \begin{pmatrix}
        x' \\
        q' 
    \end{pmatrix} - \begin{pmatrix}
        z'\\
        r'
    \end{pmatrix} \ ,
    \begin{pmatrix}
        x \\
        \nabla h(q) 
    \end{pmatrix} - \begin{pmatrix}
        z\\
        \nabla h(r)
    \end{pmatrix} \right\rangle \\
    =
    \left\langle
    \begin{pmatrix}
        x' - z'\\
        q' - r'
    \end{pmatrix} \ ,
    \begin{pmatrix}
        x'+\lambda J_\ell(x')^\intercal q' - z'- \lambda J_\ell(z')^\intercal r'\\
        \nabla h(q') - \lambda \ell(x') - \nabla h(r') + \lambda \ell(z')
    \end{pmatrix} \right\rangle \\
    =
    \|x'-z'\|^2 +
    \langle q' - r', \nabla h(q') - \nabla h(r') \rangle \\
    + \lambda \langle x'-z' , J_\ell(x')^\intercal q' - J_\ell(z')^\intercal r'  \rangle
    + \lambda \langle q'-r' , - \ell(x') + \ell(z')   \rangle \\
    \ge \|x'-z'\|^2 +
    \langle q' - r', \nabla h(q') - \nabla h(r') \rangle
    = \left\langle \begin{pmatrix}
        x' \\ q'
    \end{pmatrix}
    - \begin{pmatrix}
        z' \\ r'
    \end{pmatrix} , 
    \nabla f(x',q') - \nabla f(z',r') \right\rangle
\end{multline}
where the inequality comes from Eq. (\ref{eq:Amonotone}).

\end{proof}

We highlight that Theorem \ref{thm:bfne} generalizes the fact that the classic proximal operator is firmly nonexpansive, since $D_f$ reduces to the squared Euclidean Bregman divergence in the convex scenario $S=1$. 
Moreover, the next result shows that our prox can also be written as a generalized resolvent.
Indeed, we recall from \cite{eckstein1993nonlinear}-\cite{bauschke2003bregman}-\cite{borwein2011characterization} that an $f$-resolvent is equal to $(\nabla f + \lambda A)^{-1} \circ \nabla f$ for some monotone operator $A$. This definition extends the classic notion of resolvent, namely $(I+\lambda A)^{-1}$ (which corresponds to the particular case $f=\frac{\|\cdot\|^2}{2}$), to a general Bregman divergence $D_f$.
\begin{proposition}[$f$-resolvent]
\label{prop:resolvent}
Consider the notations introduced in Definition \ref{def:prox}.
\begin{enumerate}[(i)]
    \item The operator $A(x,q) = \begin{pmatrix}
    J_\ell(x)^\intercal q \\ -\ell(x)
\end{pmatrix}$ is monotone.
\item Our prox operator is an $f$-resolvent:
    \begin{equation*}
        \text{prox}_{\lambda \ell} = \left(\nabla f + \lambda A - 
        \begin{pmatrix}
            \mathbf{0}_m \\
            \text{LSE}(\nabla h -\lambda \ell) \mathbf{1}_S
        \end{pmatrix}
        \right)^{-1} \circ \left( \nabla f - 
        \begin{pmatrix}
            \mathbf{0}_m \\
            \text{LSE}(\nabla h) \mathbf{1}_S
        \end{pmatrix}
        \right) \ ,
    \end{equation*}
    with $A$ from (i) and $f$ from Definition \ref{def:bregdiv} and $\text{LSE}(\xi)=\log(\sum_s e^{\xi_s-1})$.
\end{enumerate}
\end{proposition}

\begin{proof}
    (i) follows from the inequality in Eq. (\ref{eq:Amonotone}) while (ii) derives from the stationarity condition (Eq. \ref{eq:statioCond}) of the saddle point $(x',q')=\text{prox}_{\lambda \ell}(x,q)$ of the function $H_{x,q}$.
\end{proof}

\paragraph{PPA and fixed points}
Theorem \ref{thm:bfne} implies that the generalized proximal point algorithm $(x^{k+1}, q^{k+1}) = \text{prox}_{\lambda \ell}(x^k,q^k)$ converges to a fixed point $(x^*,q^*)$ of the prox, if there exists any.
Such a fixed point is characterized by:
\begin{equation}
    \label{eq:fixpt}
    (x^*,q^*) = \text{prox}_{\lambda \ell}(x^*,q^*)
    \Leftrightarrow 
    \begin{cases}
        J_\ell(x^*)^\intercal q^* = 0 \\
        q^*_s = \frac{ q^*_s e^{-\lambda \ell_s(x^*)} }{ \sum_{s'} q^*_{s'} e^{-\lambda \ell_{s'}(x^*)} } \quad (\forall 1\le s \le S)
    \end{cases}
\end{equation}
which means that the $q^*$-barygradient of $\ell$ at $x^*$ is equal to zero and that: 
\begin{equation}
\label{eq:fixpt2}
    \ell_1(x^*)=\dots=\ell_S(x^*) .
\end{equation}

\section{Critical points}
\label{sec:crit}

An immediate consequence of Equations \ref{eq:fixpt}-\ref{eq:fixpt2} is that the set of fixed points $\text{Fix}(\text{prox}_{\lambda \ell})$ of our generalized proximal operator $f$-mirrors the set of critical points of the function: $\forall (x,\bar \xi)\in \mathbb{R}^m\times \mathbb{R}^{S-1} $ and $\xi = (\bar \xi,0) \in \mathbb{R}^S$ ,
\begin{equation}
     \bar F(x,\bar \xi) = \sum_{s=1}^{S-1} \bar \sigma_s(\bar \xi) \ell_s(x)  + \left(1-\sum_{s=1}^{S-1} \bar \sigma_s(\bar \xi)\right) \ell_S(x)
     = \sigma(\xi)^\intercal \ell(x) \, ,
\end{equation}
where
\begin{equation}
    \begin{pmatrix}
         \bar \sigma(\bar \xi) \\ \frac{1}{1+\sum_{s=1}^{S-1} e^{\bar \xi_s}}
     \end{pmatrix} = 
     \sigma(\xi) 
     \, .
\end{equation}

Before looking at the critical points of $\bar F$, let us first derive the expression of the Fisher information matrix (FIM) of our family of discrete distributions (see e.g. \cite{amari1998natural},\cite{amari2016information}).
\begin{proposition}[FIM]
    The Fisher information matrix $\mathbf{I}(\bar\xi)$ of the family of distributions $\{\sigma(\xi) : \xi=(\bar \xi,0), \bar\xi \in \mathbb{R}^{S-1} \}$ is given by:
    \begin{equation*}
    \forall 1\le i,j\le S-1, \quad [\mathbf{I}(\bar \xi)]_{i,j} = \delta_{i,j} \bar \sigma_i(\bar \xi) - \bar \sigma_i(\bar \xi) \bar \sigma_j(\bar \xi) ,
\end{equation*}
with $\delta_{i,j}$ denoting Kronecker deltas.
Equivalently,
\begin{equation*}
    \mathbf{I}(\bar \xi) = \text{Diag}(\bar \sigma(\bar \xi)) - \bar \sigma(\bar \xi) \bar \sigma(\bar \xi)^\intercal \quad
    \text{and} \quad \mathbf{I}(\bar \xi)^{-1} = \text{Diag}(\bar \sigma(\bar \xi))^{-1} + \frac{1}{1-\sum_{s=1}^{S-1} \bar \sigma_s(\bar \xi)} \mathbf{1}_{S-1} \mathbf{1}_{S-1}^\intercal .
\end{equation*}
\end{proposition}

\begin{proof}
    By definition:
    \begin{multline*}
        [\mathbf{I}(\bar \xi)]_{i,j} = \mathbb{E}_{\tau \sim \sigma(\xi)}\left[ \frac{\partial \log \sigma_\tau(\xi)}{\partial \bar \xi_i} \frac{\partial \log \sigma_\tau(\xi)}{\partial \bar \xi_j}   \right]
        = \sum_{\tau=1}^S \sigma_\tau(\xi) \left[ \frac{\partial \log \sigma_\tau(\xi)}{\partial \bar \xi_i} \frac{\partial \log \sigma_\tau(\xi)}{\partial \bar \xi_j}   \right] .
    \end{multline*}
    For all $1\le i\le S-1,1\le \tau\le S$, we have:
    \begin{equation*}
        \frac{\partial \log \sigma_\tau(\xi)}{\partial \bar \xi_i} = \delta_{i,\tau} - \bar \sigma_i(\bar \xi).
    \end{equation*}
    Hence,
    \begin{multline*}
         \frac{\partial \log \sigma_\tau(\xi)}{\partial \bar \xi_i} \frac{\partial \log \sigma_\tau(\xi)}{\partial \bar \xi_j}
         = \left[\delta_{i,\tau} - \bar \sigma_i(\bar \xi)\right]
         \left[\delta_{j,\tau} - \bar \sigma_j(\bar \xi)\right] \\
         = \delta_{i,\tau} \delta_{j,\tau} 
         - \delta_{i,\tau} \bar \sigma_j(\bar \xi)
         - \delta_{j,\tau} \bar \sigma_i(\bar \xi)
         + \bar \sigma_i(\bar \xi) \bar \sigma_j(\bar \xi) ,
    \end{multline*}
    from which we deduce that 
    \begin{equation*}
        [\mathbf{I}(\bar \xi)]_{i,j} = \delta_{i,j} \bar \sigma_i(\bar \xi) - \bar \sigma_i(\bar \xi) \bar \sigma_j(\bar \xi) .
    \end{equation*}
\end{proof}

In particular the FIM defines the Fisher-Rao metric (\cite{rao1992information}).

Given that

\begin{equation}
    \nabla \bar F(x,\bar \xi) = \begin{pmatrix}
        J_\ell(x)^\intercal \sigma(\xi) \\
        \mathbf{I}(\bar \xi) \bar \ell(x)
    \end{pmatrix}
    \quad \text{with} \quad \bar \ell = \begin{pmatrix}
        \ell_1-\ell_S \\ \vdots \\ \ell_{S-1}-\ell_S
    \end{pmatrix} ,
\end{equation}
we can now state our result concerning the set of critical points of $\bar F$.

\begin{theorem} 
\label{thm:fixprox}
Recall that for any $\bar \xi \in \mathbb{R}^{S-1}$, we denote $\xi=(\bar \xi, 0)\in \mathbb{R}^S$. Then,
    \begin{equation*}
    \text{Fix}(\text{prox}_{\lambda \ell}) = 
    \{ (x,\sigma(\xi)) : \nabla \bar F(x,\bar \xi) = 0 \}.
\end{equation*}
\end{theorem}

\begin{proof}
    We have:
    \begin{multline}
        \nabla \bar F(x,\xi) = \begin{pmatrix}
            J_\ell(x)^\intercal \sigma(\xi) \\ \mathbf{I}(\bar \xi) \bar \ell(x)
        \end{pmatrix} = 0 \Leftrightarrow \begin{cases}
            J_\ell(x)^\intercal \sigma(\xi) =0 \\
            \bar \ell(x) \in \text{ker}(\mathbf{I}(\bar \xi))=\{\mathbf{0}_{S-1}\}
            \end{cases} \\
            \Leftrightarrow (x,\sigma(\xi)) \in \text{Fix}(\text{prox}_{\lambda \ell}).
    \end{multline}
\end{proof}

We now provide a result concerning the critical values of $\bar F$, i.e. the values taken at critical points.
\begin{proposition}
\label{prop:critval}
    All critical values of $\bar F$ are identical.
\end{proposition}

\begin{proof}
    Assume $(x_1,\bar \xi_1), (x_2,\bar \xi_2) \in \mathbb{R}^{m+S-1}$ are two critical points of $\bar F$: for $i\in\{1,2\}$,
    \begin{equation*}
        \nabla \bar F(x_i,\bar \xi_i) = \begin{pmatrix}
            J_\ell(x_i)^\intercal \sigma(\xi_i) \\
            \mathbf{I}(\bar \xi_i) \bar \ell(x_i)
        \end{pmatrix} = 0 .
    \end{equation*}
    Thus $x_i$ minimizes the convex function $x\mapsto \sigma(\xi_i)^\intercal \ell(x_i)$, and we know from Theorem \ref{thm:fixprox} that
    \begin{equation}
        \ell_1(x_i) = \dots = \ell_S(x_i) .
    \end{equation}
    We deduce that:
    \begin{equation}
    \label{eq:ineqFtilde}
        \bar F(x_1,\bar \xi_1) \le \bar F(x_2,\bar \xi_1) = \sigma(\xi_1)^\intercal \ell(x_2)
        = \sigma(\xi_2)^\intercal \ell(x_2)
        = \bar F(x_2,\bar \xi_2) ,
    \end{equation}
    and similarly that $\bar F(x_2,\bar \xi_2) \le \bar F(x_1,\bar \xi_1)$.
\end{proof}

Proposition \ref{prop:critval} generalizes the property that a convex function has at most a unique critical value.
Furthermore, Eq. (\ref{eq:ineqFtilde}) shows that if there exists $s\in\{1,\dots,S\}$ such that $\ell_s$ is strictly convex, then necessarily $x_1=x_2$.

\paragraph{Riemannian Hessian matrix}
From now on, let us assume that every $\ell_s$ is twice continuously differentiable.
For $p\in \mathbb{R}^{S-1}$, denote by $\mathbf{T}(p) \in \mathbb{R}^{(S-1)\times (S-1)\times (S-1)}$ the tensor derivative of $\mathbf{C}(p) = \text{Diag}(p) - p p^\intercal$ given by
\begin{equation}
    \forall (i,j,k)\in \{1,\dots,S-1\}^3 \ , \quad [\mathbf{T}(p)]_{i,j,k} = \frac{\partial [\mathbf{C}(p)]_{i,j}}{\partial p_k} = \delta_{i,j} \delta_{i,k} - p_j \delta_{i,k} - p_i \delta_{j,k} \, .
\end{equation}
The Hessian matrix of $\bar F$ admits the following expression:
\begin{equation}
\label{eq:EuclHess}
    \nabla^2 \bar F(x,\bar \xi) = \begin{pmatrix}
        \sum_s \sigma_s(\xi) \nabla^2 \ell_s(x) & J_{\bar \ell}(x)^\intercal \mathbf{I}(\bar \xi) \\
        \mathbf{I}(\bar \xi) J_{\bar \ell}(x) & [\mathbf{T}(\bar \sigma(\bar \xi)) \times_2 \bar \ell(x) ] \mathbf{I}(\bar \xi)
    \end{pmatrix} ,
\end{equation}
where the matrix with entries $[\mathbf{T}(p) \times_2 v]_{i,k}=\sum_j [\mathbf{T}(p)]_{i,j,k} v_j$ is the second mode contraction of tensor $\mathbf{T}(p)$ with some vector $v$.

Letting
\begin{equation}
    \mathbf{M}(x,\bar \xi) =
    \begin{pmatrix}
        I_m & \mathbf{0}_{m\times (S-1)} \\
        \mathbf{0}_{(S-1)\times m} & \mathbf{I}(\bar \xi)
    \end{pmatrix}
\end{equation}
be the $x$-Euclidean $\bar \xi$-Fisher-Rao product Riemannian metric,
we deduce that the Riemannian Hessian $\widetilde{\nabla}^2 \bar F$ with respect to $\mathbf{M}$ is given by:
\begin{equation}
\label{eq:RiemHess}
    \widetilde{\nabla}^2 \bar F(x,\bar \xi) = \nabla^2 \bar F(x,\bar \xi)
    -\begin{pmatrix}
        \mathbf{0}_{m\times m} & \mathbf{0}_{m\times (S-1)} \\
        \mathbf{0}_{(S-1)\times m} & \begin{bmatrix}
            \sum_k \Gamma^k_{i,j}(\bar \xi) \frac{\partial \bar F}{\partial \bar \xi_k}(x,\bar \xi)
        \end{bmatrix}_{i,j}
    \end{pmatrix}
\end{equation}
with Christoffel symbols provided below.
\begin{proposition}[Christoffel symbols]
\label{prop:Christoffel}
    \begin{equation*}
    \Gamma^k_{i,j}(\bar\xi) = 
    \frac12 \left\{ \delta_{i,j}\delta_{i,k} - \delta_{i,j} \bar \sigma_i(\bar \xi) - \delta_{i,k} \bar \sigma_j(\bar \xi) - \delta_{j,k} \bar \sigma_i(\bar \xi) + 2 \bar \sigma_i(\bar \xi) \bar \sigma_j(\bar \xi) \right\}.
\end{equation*}
\end{proposition}

\begin{proof}
    We consider the distributions $\sigma(\xi)$ that form an exponential family with potential function $\psi(\bar\xi)=\log(1+\sum_{s=1}^{S-1} e^{\bar \xi_s})$ (see page 80 in \cite{amari2012differential}). Recall that $[\mathbf{I}(\bar \xi)]_{i,j}=\frac{\partial^2 \psi}{\partial \bar \xi_i \partial \bar \xi_j} (\bar \xi)$, from which we get the Christoffel symbols of the first kind (see pages 41 and 106 in \cite{amari2012differential}):
    \begin{multline*}
        \Gamma_{i,j,k}(\bar \xi) = \frac12 \frac{\partial^3 \psi}{\partial \bar \xi_i \partial \bar \xi_j \partial \bar \xi_k}(\bar \xi)
        = \frac12 \frac{\partial [\mathbf{I}(\bar \xi)]_{i,j}}{\partial \bar \xi_k} \\
        = \frac12 \left\{ \delta_{i,j} \delta_{i,k} \bar \sigma_i(\bar \xi) -\delta_{i,j} \bar \sigma_i(\bar \xi) \bar \sigma_k(\bar \xi) -\delta_{i,k} \bar \sigma_i(\bar \xi) \bar \sigma_j(\bar \xi) -\delta_{j,k} \bar \sigma_i(\bar \xi) \bar \sigma_j(\bar \xi)
        + 2 \bar \sigma_i(\bar \xi) \bar \sigma_j(\bar \xi) \bar \sigma_k(\bar \xi)
        \right\} .
    \end{multline*}
    We deduce (see e.g. \cite{weisstein2003christoffel}) that
    \begin{multline*}
        \Gamma^k_{i,j}(\bar \xi)
        = \sum_{s=1}^{S-1} [\mathbf{I}(\bar \xi)^{-1}]_{k,s} \Gamma_{i,j,s}(\bar \xi) = \\
         \frac12 \left\{ \delta_{i,j}\delta_{i,k} - \delta_{i,j} \bar \sigma_i(\bar \xi) - \delta_{i,k} \bar \sigma_j(\bar \xi) - \delta_{j,k} \bar \sigma_i(\bar \xi) + 2 \bar \sigma_i(\bar \xi) \bar \sigma_j(\bar \xi) \right\} .
    \end{multline*}
\end{proof}

\begin{remark}[Potential function]
    One easily verifies that the Riemannian Hessian matrix of $(x,\bar \xi) \mapsto \frac12 \|x\|^2 + \psi(\bar \xi)$ with the log-partition function $\psi(\bar \xi)=\log(1+\sum_{s=1}^{S-1} e^{\bar \xi_s})$ is exactly $\mathbf{M}(x,\bar \xi)$. Indeed, the correction term involving the Christoffel symbols vanishes (namely $\sum_k \Gamma^k_{i,j} \frac{\partial \psi}{\partial \bar \xi_k}=0$).
\end{remark}

\begin{proposition}[Riemannian Hessian]
\label{prop:RiemHess}
    The Riemannian Hessian of $\bar F$ with respect to $\mathbf{M}$ is equal to
    \begin{multline*}
    \widetilde{\nabla}^2 \bar F(x,\bar \xi) =
    \begin{pmatrix}
        \sum_s \sigma_s(\xi) \nabla^2 \ell_s(x) & J_{\bar \ell}(x)^\intercal \mathbf{I}(\bar \xi) \\
        \mathbf{I}(\bar \xi) J_{\bar \ell}(x) & \mathbf{H}(x,\bar \xi)
    \end{pmatrix} \, , \\
    \text{where} \quad \mathbf{H} = \frac12 \left\{ \text{Diag}(\bar \sigma)\mathbf{D} - \mathbf{D} \bar \sigma \bar \sigma^\intercal - \bar \sigma \bar \sigma^\intercal \mathbf{D} \right\} \text{ with } \mathbf{D}=\text{Diag}(\bar \ell - \mathbf{1}_{S-1} \bar\sigma^\intercal \bar \ell) \, .
\end{multline*}
\end{proposition}

\begin{proof}
    By combining Eqs. \ref{eq:EuclHess} and \ref{eq:RiemHess} with Proposition \ref{prop:Christoffel} and noticing that
    \begin{equation}
        [\mathbf{T}(\bar \sigma(\bar \xi)) \times_2 \bar \ell(x) ] \mathbf{I}(\bar \xi) = 2 \begin{bmatrix}
            \sum_k \Gamma^k_{i,j}(\bar \xi) \frac{\partial \bar F}{\partial \bar \xi_k}(x,\bar \xi)
        \end{bmatrix}_{i,j}
        = 2 \mathbf{H}(x,\bar\xi) .
    \end{equation}
\end{proof}

\begin{remark}[Inertia and saddle point]
\label{rmk:eigen}
Proposition \ref{prop:RiemHess} shows that the Riemannian Hessian of $\bar F$ is in general not positive semi-definite (by using Haynsworth inertia additivity formula if at least one $\ell_s$ satisfies $\nabla^2 \ell_s \succ \mathbf{0}$\footnote{Which implies that $\ell_s$ is strictly convex (see section 3.1.4 in \cite{boyd2004convex}) and $\mathbf{B}_1 \succ \mathbf{0}$ by Weyl's inequality.}) and hence that $\bar F$ is not necessarily geodesically convex (see \cite{udriste2013convex} or Theorem 11.23 in \cite{boumal2023introduction}).
More precisely, $\widetilde{\nabla}^2 \bar F(x,\bar \xi)$ is congruent to (and thus, by Sylvester's law of inertia, has the same inertia as) the block-diagonal matrix $\text{Diag}(\mathbf{B}_1,\mathbf{B}_2)$ with positive definite matrix $\mathbf{B}_1=\sum_s \sigma_s(\xi) \nabla^2 \ell_s(x)$ and its Schur complement $\mathbf{B}_2 = \mathbf{H}(x,\bar \xi) - \mathbf{I}(\bar \xi) J_{\bar \ell}(x) \mathbf{B}_1^{-1} J_{\bar \ell}(x)^\intercal \mathbf{I}(\bar \xi)$.
In particular, for $(x,\bar\xi)=(x^*,\bar\xi^*)$ a critical point of $\bar F$ we have $\mathbf{H}(x^*,\bar\xi^*)=\mathbf{0}$ and $\mathbf{B}_2 \preceq \mathbf{0}$ from which we deduce that $(x^*,\bar\xi^*)$ must be a saddle point.
\end{remark}





\section{Barygradient flows}

\subsection{Barygradient min-max flow}

Now let us generalize the gradient flow ordinary differential equation (ODE) by letting $\lambda \rightarrow 0$ in our generalized PPA.

\begin{definition}
Let $F(x,q) = q^\intercal \ell(x)$. We define the barygradient min-max flow ODE as
\begin{equation*}
\dot \zeta(t) = - \begin{pmatrix}
I_m & 0 \\
0 & -I_S
\end{pmatrix} \nabla F( (\nabla f)^{-1}( \zeta(t) -\log(\sum_s e^{\xi_s(t)-1}) \begin{pmatrix} \mathbf{0}_m \\ \mathbf{1}_S \end{pmatrix} ) ) + \gamma(t) \begin{pmatrix} \mathbf{0}_m \\ \mathbf{1}_S \end{pmatrix} ,
\end{equation*}
where $\zeta = (x,\xi) : \mathbb{R}_+ \rightarrow \mathbb{R}^m \times \mathbb{R}^S$ and
\begin{equation*}
\gamma(t) = \frac{\sum_s [ \dot \xi_s(t) - \ell_s(x(t)) ] e^{\xi_s(t)-1}}{\sum_s e^{\xi_s(t)-1}} = q(t)^\intercal [ \dot \xi(t) - \ell(x(t)) ]
\end{equation*}
with $q(t)=\sigma(\xi(t))$.
\end{definition}

We point out that 
\begin{equation*}
\begin{pmatrix}
I_m & 0 \\
0 & -I_S
\end{pmatrix} \nabla F( (\nabla f)^{-1}( \zeta(t) -\log(\sum_s e^{\xi_s(t)-1}) \begin{pmatrix} \mathbf{0}_m \\ \mathbf{1}_S \end{pmatrix}) ) = A(x(t), q(t)) .
\end{equation*}

\paragraph{Monotonicity analysis}

Contrary to classic gradient flow, the function $F(x(t),q(t))$ is not necessarily nonincreasing along the flow.
Indeed,
\begin{multline*}
\frac{d}{dt}F( (\nabla f)^{-1}( \zeta(t) -\log(\sum_s e^{\xi_s(t)-1}) \begin{pmatrix} \mathbf{0}_m \\ \mathbf{1}_S \end{pmatrix} ) ) =\\ \frac{d}{dt}[(\nabla h)^{-1}(\xi(t)-\log(\sum_s e^{\xi_s(t)-1}) \mathbf{1}_S)]^\intercal \ell(x(t)) + q(t)^\intercal \frac{d}{dt} \ell(x(t)) ,
\end{multline*}
where 
\begin{equation*}
\frac{d}{dt}[(\nabla h)^{-1}(\xi(t)-\log(\sum_s e^{\xi_s(t)-1}) \mathbf{1}_S)] = [\nabla^2 h(q(t))]^{-1} \dot \xi(t) - \frac{\sum_s \dot \xi_s(t) e^{\xi_s(t)-1}}{\sum_s e^{\xi_s(t)-1}} q(t)
\end{equation*}
and $\frac{d}{dt} \ell(x(t)) = J_\ell(x(t)) \dot x(t)$.
Hence,
\begin{equation*}
\frac{d}{dt}F(x(t),q(t)) = \underbrace{ \ell(x(t))^\intercal [\nabla^2 h(q(t))]^{-1} \ell(x(t)) - F(x(t),q(t))^2 }_{\text{Var}_{\tau \sim q(t)}(\ell_\tau(x(t)))} - ||J_\ell(x(t))^\intercal q(t)||^2 ,
\end{equation*}
which is not necessarily nonpositive.

\paragraph{Entropy dynamics}

Denote $\chi(t) = h(q(t))$. Then,
\begin{multline*}
\dot \chi(t) = \dot q(t)^\intercal \nabla h(q(t))
= \{ [\nabla^2 h(q(t))]^{-1} \dot \xi(t) - [q(t)^\intercal \dot \xi(t)] q(t) \}^\intercal \{\xi(t)-\log(\sum_s e^{\xi_s(t)-1}) \mathbf{1}_S\} \\ = \xi(t)^\intercal \underbrace{[ \text{Diag}(q(t))-q(t)q(t)^\intercal ]}_{\text{Cov}(q(t))} \ell(x(t)) ,
\end{multline*}
where $\text{Cov}(q(t))$ denotes the covariance matrix\footnote{$\text{Cov}(q(t))$ is also the Jacobian matrix of $\sigma$ evaluated at $\xi(t)$.} of the categorical distribution $q(t)$.

\begin{remark}
The barygradient flow can be equivalently rewritten as the following pseudo-Riemannian flow:
\begin{equation*}
\dot \zeta(t) = - \begin{pmatrix}
I_m & 0 \\
0 & -\text{Cov}(q(t))^\dagger
\end{pmatrix} \nabla \tilde F( \zeta(t) ) + [ \gamma(t) + \frac{\mathbf{1}_S^\intercal \ell(x(t))}{S} ] \begin{pmatrix} \mathbf{0}_m \\ \mathbf{1}_S \end{pmatrix} ,
\end{equation*}
where $\dagger$ denotes the Moore–Penrose pseudoinverse (see \cite{meyer1973generalized}) and
\begin{equation*}
\tilde F(x,\xi) = \sigma(\xi)^\intercal \ell(x) .
\end{equation*}
\end{remark}

\subsection{Barygradient min-min flow}

Similarly, we define the barygradient min-min flow as follows.

\begin{definition}
Let $F(x,q) = q^\intercal \ell(x)$. We define the barygradient min-min flow ODE as
\begin{multline*}
\dot \zeta(t) = - \nabla F( (\nabla f)^{-1}( \zeta(t) -\log(\sum_s e^{\xi_s(t)-1}) \begin{pmatrix} \mathbf{0}_m \\ \mathbf{1}_S \end{pmatrix} ) ) + \gamma(t) \begin{pmatrix} \mathbf{0}_m \\ \mathbf{1}_S \end{pmatrix} \\
= - \begin{pmatrix}
I_m & 0 \\
0 & \text{Cov}(q(t))^\dagger
\end{pmatrix} \nabla \tilde F( \zeta(t) ) + [ \gamma(t) - \frac{\mathbf{1}_S^\intercal \ell(x(t))}{S} ] \begin{pmatrix} \mathbf{0}_m \\ \mathbf{1}_S \end{pmatrix} ,
\end{multline*}
where
\begin{equation*}
\gamma(t) = q(t)^\intercal [ \dot \xi(t) + \ell(x(t)) ] .
\end{equation*}
\end{definition}

As in the min-max case, we can study the dynamics of $F(x(t),q(t))$ and the negentropy $h(q(t))$.
Indeed, we have:
\begin{equation*}
\frac{d}{dt}F(x(t),q(t)) = - \underbrace{\text{Var}_{\tau \sim q(t)}(\ell_\tau(x(t))) }_{\ell(x(t))^\intercal \text{Cov}(q(t)) \ell(x(t))} - ||J_\ell(x(t))^\intercal q(t)||^2 \le 0 ,
\end{equation*}
which shows that $F(x(t),q(t))$ is nonincreasing.
For the negentropy we have:
\begin{equation*}
\frac{d}{dt} h(q(t)) = - \xi(t)^\intercal \text{Cov}(q(t)) \ell(x(t)) .
\end{equation*}

\newpage

\subsubsection*{Acknowledgments}
The author thanks Adil Salim and Massil Achab for helpful discussions on convex analysis.

\bibliography{bib}

\end{document}